\documentclass[11pt]{amsart}

\usepackage{amssymb}

\usepackage{epsfig,amsfonts,amssymb,amsthm}

\theoremstyle{plain}
\newtheorem{theorem}{Theorem}
\newtheorem{lemma}{Lemma}
\newtheorem{corollary}{Corollary}

\theoremstyle{definition}

\newtheorem{remark}{Remark}

\theoremstyle{plain}
\newtoks\thehProclaim
\newtheorem*{Proclaim}{\the\thehProclaim}

\theoremstyle{definition}
\newtoks{\thehRemark}
\newtheorem*{Remark}{\the\thehRemark}

\renewcommand{\leq}{\leqslant}
\renewcommand{\geq}{\geqslant}

\begin{document}

\dedicatory{To Vassili Mikhailovich Babich on occasion of his 90-th birthday}

\title[Searchlight for Boundary Inflection]
{Searchlight Asymptotics for High-Frequency Scattering by Boundary Inflection}

\author{V.P. Smyshlyaev}

\address{Department of Mathematics\\
University College London \\
Gower Street \\
London\\
WC1E 6BT\\
UK}

\email{v.smyshlyaev@ucl.ac.uk}

\author{I.V. Kamotski}

\address{Department of Mathematics\\
University College London \\
Gower Street \\
London\\
SW1E 6BT\\
UK}

\email{i.kamotski@ucl.ac.uk}



\begin{abstract}
We consider an inner problem for whispering gallery high-frequency asymptotic mode's scattering by a boundary inflection.   
The related boundary-value problem for a Schr\"{o}dinger equation on a half-line with a potential linear in both space and time 
appears fundamental for describing transitions from modal to scattered asymptotic patterns, and despite having been intensively studied over several decades remains largely unsolved. 
We prove that the solution past the inflection point has a ``searchlight'' asymptotics corresponding to a beam 
concentrated near the limit ray, and establish certain decay and smoothness properties of the related searchlight amplitude. 
We also discuss further interpretations of the above result: the existence of associated generalised wave operator, and of a  version of a unitary scattering operator connecting the modal and scattered asymptotic regimes. 
\end{abstract}

\thanks{}


\maketitle

\section{Introduction}

This work is dedicated to Prof V.M. Babich on occasion of his 90-th birthday, and  is devoted to a long-standing canonical problem of high frequency diffraction, to which area Prof Babich has made many groundbreaking contributions. 
The particular problem is an inner problem for a whispering gallery high-frequency asymptotic mode 
propagating along a concave part of a boundary and then scattering at a boundary inflection point. 
Like Airy equation and associated Airy function are fundamental for describing transition from oscillatory 
to exponentially decaying asymptotic behaviors, the boundary inflection problem leads to an arguably 
equally fundamental canonical boundary-value problem for a special partial differential equation (PDE) 
describing transition from a ``modal'' to a ``scattered'' high-frequency asymptotic behavior. 

Mathematically, the problem is for a Schr\"{o}dinger equation on a half-line with a potential linear in both space and time, which was first formulated and analysed by M.M. Popov starting from \cite{P79} in the 1970-s. The associated solutions 
have asymptotic behaviors with a discrete spectrum at one end and with a continuous spectrum at the other end, and of central 
interest is to find the map connecting the above two asymptotic regimes. The problem however lacks 
separation of variables, except in the asymptotical sense at both of the above ends. 
Nevertheless, as we essentially argue in the present work, a non-standard perturbation analysis at 
the continuous spectrum end can be performed, ultimately describing the desired map connecting 
the two asymptotic representations. 

Specifically, the problem is for a high-frequency wave process near a boundary containing a simple inflection point, as 
displayed on Fig. \ref{fig1}. 
\begin{figure}
	\centering
		\includegraphics[scale=0.55]{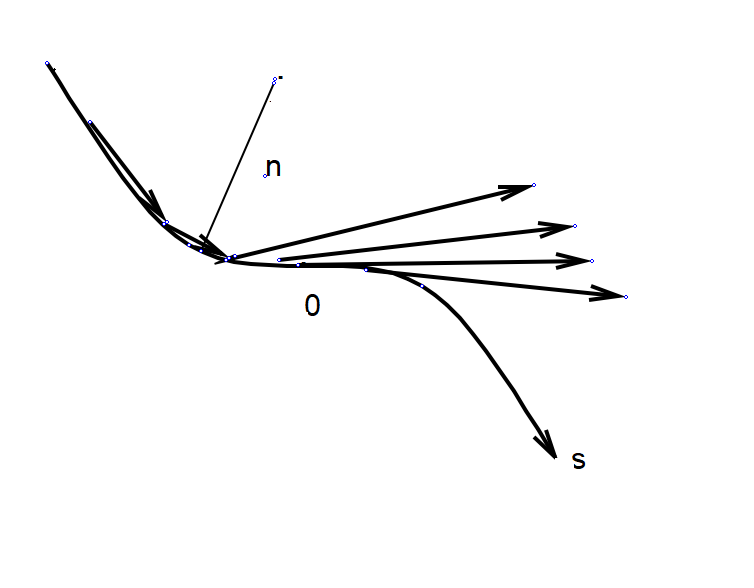}
		\caption{Whispering gallery wave near a boundary inflection } 
		\label{fig1} 
\end{figure} 
Let the wave process be described by Helmholtz equation $\Delta U+k^2U=0$ over the boundary $S$ and let $(s,n)$ be a local system 
of coordinates associated with the boundary, so that $s$ is the arclength along $S$ and $n$ is the normal distance to $S$. So the 
boundary corresponds to $n=0$, with (say) Dirichlet  boundary conditions $U(s,0)=0$. The inflection point corresponds to $s=0$ and 
so the curvature $\kappa(s)$ is positive for $s<0$ where the boundary is (locally) concave and negative for $s>0$ where the boundary 
becomes convex. 

For the concave part of the boundary, i.e. for $s<0$ away from the inflection, whispering gallery asymptotic solutions are known to exist for large $k$ (i.e for high frequencies) propagating along $S$ in a thin ``boundary-layer'' $\,n=O\left(k^{-2/3} \right)$. To the main order, these are of the form, see e.g. \cite{BB} and \cite{BK}, 
\begin{equation}
U\sim C\kappa^{1/6}(s)e^{iks}\exp\left\{-i2^{-1/3}k^{1/3}\nu_j\int_{s_0}^s\kappa^{2/3}(s')ds'\right\}
\mbox{Ai}\left(\left(2\kappa(s)\right)^{1/3}k^{2/3}n-\nu_j\right), 
\label{WGincom}
\end{equation}
where $\mbox{Ai}(z)$ is the Airy function, solution to Airy equation $\mbox{Ai}^{''}-z\mbox{Ai}=0$ exponentially decaying as $z\to +\infty$; $-\nu_j$, $j=1,2,3,...$, is any of its zeros, 
$0<\nu_1<\nu_2<...<\nu_j<\nu_{j+1}<...$, $\nu_j\to +\infty$ as $j\to\infty$; $s_0<0$ is any reference point,  
and $C$ is a constant. 

For the curvature $\kappa(s)=-\beta s+O(s^2)$ near the inflection point $s=0$, $\beta>0$, Popov has shown in \cite{P79} that the 
incoming whispering gallery wave \eqref{WGincom} has to be matched with a solution of an inner problem near $s=0$, which in 
stretched coordinates 
$x=\beta^{1/5}k^{3/5}n$, $t=\beta^{2/5}k^{1/5}s$, with $U(s,n)\sim e^{iks}\psi(x,t)$ is stated in \eqref{PopEq}--\eqref{PopAC} 
below. 

Since Popov's pioneering work the problem has attracted a sustained attention over the last several decades. Without attempting a comprehensive review, we discuss in the next section in some detail a number of  contributions particularly relevant to our 
present purposes. Additionally, Kazakov has introduced and  analysed asymptotically certain integral solutions to \eqref{PopEq} in \cite{kaz1}, and most recently proposed in \cite{kaz2} another approach for constructing a class of particular solutions to 
\eqref{PopEq} obeying boundary conditions \eqref{PopBC}. 

In the present work, we aim at rigorously establishing the structure of a ``searchlight'' formed past the inflection point, which corresponds to the asymptotics of the solution to \eqref{PopEq}--\eqref{PopAC} as $t\to +\infty$ near 
the limit ray corresponding to the straight line tangent to $S$ at $s=0$, see Fig. \ref{fig1}. Our new result is stated in Theorem \ref{thm1}, which 
justifies to the main order a related asymptotic ansatz \eqref{SLansatz} 
and establishes some properties of the searchlight amplitude $G_0(\eta)$. 

It appears that mathematically the searchlight asymptotics corresponds to a perturbation analysis as $t\to+\infty$ of  
an appropriately transformed problem. 
The underlying key transformation, first discovered in \cite{BS1}, turns out to be closely related to a ``pseudoconformal 
symmetry'' of a free Schr\"{o}dinger equation, see e.g. \cite{Talanov,Sulem,Tao}. 
The related unperturbed problem appears to be one with a continuous spectrum, corresponding to a ``free scattering'' without a boundary. This can be compared with Popov's asymptotic analysis as $t\to -\infty$ in \cite{P82}, where another transformed problem is a perturbation to an operator with a discrete spectrum corresponding to the whispering gallery modes. 
Wave operators appear to exist at both ``ends'' $t\to \pm\infty$, and hence so does a version of a scattering operator connecting 
the two asymptotic regimes, see Theorem \ref{thm2} and Corollary \ref{cor1}. 
We remark that the general idea of transforming an operator consistently with the form of the solutions' asymptotic expansion appears to resonate with that of generalised wave operators, cf \cite{Buslaev}. 
Notice also that other scenarios connecting asymptotic regimes with discrete and continuous spectra are found in the literature, see e.g. 
\cite{FS}. 

In the next section, we formulate the inner problem for $\psi(x,t)$, discuss related background and state our main result, 
Theorem \ref{thm1}. In Section \ref{sect3} we discuss interpretations for the searchlight wave operator and for a version 
of a scattering operator. 
Section \ref{proofthm1} provides a detailed proof of Theorem \ref{thm1}.

\section{Formulation, background, and the main result}

The inner problem for whispering gallery wave scattering near a simple boundary inflection point was formulated by Popov 
\cite{P79} as 
follows. In the half-plane domain $\Omega:=\{(x,t)\,;\,x>0,\, -\infty<t<+\infty\}$, find a solution $\psi(x,t)$ to the Schr\"{o}dinger-type equation
\begin{equation}
i\frac{\partial\psi}{\partial t}\,=\,-\,\frac{1}{2}\,\frac{\partial^2\psi}{\partial x^2}\,-\,xt\,\psi, 
\label{PopEq}
\end{equation}
subject to Dirichlet boundary condition at $x=0$ 
\begin{equation}
\psi(0,t)\,=\,0, 
\label{PopBC}
\end{equation}
and with a prescribed asymptotic behaviour for large negative $t$: 
\begin{equation}
\left\|\psi(\cdot,t)\,-\,\psi^-_0(\cdot,t)\right\|_{L^2(0,+\infty)}\,\to\,0, \,\, {\mbox as }\, t\to\,-\infty, 
\label{PopAC0}
\end{equation}
where 
\begin{equation}
\psi^-_0(x,t)\,=\,D_j(-2t)^{1/6}\exp\left(i\nu_j\frac{3}{20}(-2t)^{5/3}\right)\mbox{Ai}\left(x(-2t)^{1/3}-\nu_j\right). 
\label{PopAC}
\end{equation}
Here $\mbox{Ai}(z)$ is the Airy function 
and 
$D_j$ is a constant. 
Formula \eqref{PopAC} emerges in \cite{P79} as a result of matching as $t\to -\infty$ with the principal part of asymptotics of the incoming whispering gallery wave \eqref{WGincom}. 

PDE \eqref{PopEq} when regarded on the whole real line $\mathbb{R}$ ($-\infty<x<+\infty$) is known to be reducible to a separable 
one, see e.g. \cite{BPZ}. However the boundary condition \eqref{PopBC} makes it non-separable, which is a source of major 
challenges for its analysis. 

In \cite{P82} Popov studied well-posedness of \eqref{PopEq}--\eqref{PopAC}, whose solution $\psi(\cdot,t)$ was a priori understood as  
a function from the domain of a self-adjoint operator $A(t)$ in Hilbert space 
$H=L^2(0,+\infty)$ determined by the right hand side of \eqref{PopEq} 
with zero boundary condition at $x=0$. 
Function $\psi(\cdot,t)$ is required to be strongly differentiable in $t$ in $H$ with the left 
hand side of \eqref{PopEq} understood as the strong derivative. The existence (and uniqueness) was then established in \cite{P82} 
by means of wave operator methods as $t\to -\infty$. 
(For an introduction to as well as for more advanced properties of wave operators, see e.g. \cite{Kato}, \cite{Yafaev}.) 

To that end, Popov used in \cite{P82} the following natural transformation of \eqref{PopEq}--\eqref{PopBC} for identifying an 
``unperturbed'' problem as 
$t\to-\infty$, for which \eqref{PopAC} is an exact solution. 
For $t<0$, introduce new variables $(\xi,\tau)$ and function $\tilde{\psi}$ by 
\begin{equation}
\xi\,=\,(-2t)^{1/3}x, \ \ \tau\,=\,-\frac{3}{20}(-2t)^{5/3}, \ \ \ \psi(x,t)\,=\,(-2t)^{1/6}\tilde{\psi}(\xi,\tau). 
\label{PopTransf}
\end{equation}
This transforms \eqref{PopEq} into 
\begin{equation}
i\frac{\partial\tilde\psi}{\partial \tau}\,=\,-\,\, \frac{\partial^2\tilde\psi}{\partial \xi^2}\,+\,\xi\,\psi\,\,
-\,\frac{i}{5\tau}\left(\xi\frac{\partial\tilde\psi}{\partial \xi}+\frac{1}{2}\tilde\psi\right), \ \ \xi>0, \ \tau<0, 
\label{PopEq-}
\end{equation}
with similar to \eqref{PopBC} boundary condition $\tilde\psi(0,\tau)=0$. 

The asymptotics $\psi_0^-(x,t)$ in \eqref{PopAC} then transforms into 
\begin{equation}
\tilde\psi_0^-(\xi,\tau)\,=\,D_j \exp\left(-i\nu_j\tau\right)\mbox{Ai}(\xi-\nu_j), 
\label{EF-}
\end{equation} 
which is an exact solution of the 
unperturbed version of \eqref{PopEq-} as $\tau\to -\infty$ i.e. of the one with the last term on the right hand side of 
 \eqref{PopEq-} (the perturbation) 
omitted. 
Transformation \eqref{PopTransf}--\eqref{PopEq-} also helps to recover an all-order formal asymptotic expansion to the solution of  
\eqref{PopEq}--\eqref{PopAC} as $t\to -\infty$, initially constructed in \cite{P79}: 
\begin{equation}
\psi(x,t)\,\sim\,(-2t)^{1/6}e^{-i\nu_j\tau}\sum_{n=0}^\infty\tau^{-n}
\left[
P_{2n}(\xi)\mbox{Ai}\left(\xi-\nu_j\right)+Q_{2n-1}(\xi)\mbox{Ai}'\left(\xi-\nu_j\right)
\right], 
\label{PopAsFull}
\end{equation}
where $\xi$ and $\tau$ are given by \eqref{PopTransf}; 
$\mbox{Ai}'$ is the derivative of the Airy function, and $P_{2n}$ and $Q_{2n-1}$ are polynomials in $\xi$ of degrees $2n$ and $2n-1$ respectively, $Q_{2n-1}(0)=0$ ($Q_{-1}\equiv 0$). 
Indeed, substitution of \eqref{PopAsFull} into \eqref{PopEq-} using \eqref{PopTransf} and the Airy equation leads to 
straightforward recurrence relations for $P_{2n}$ and $Q_{2n-1}$. 

Popov's well-posedness analysis in \cite{P82} used in fact a further transformation of \eqref{PopEq-} via 
$\tilde\psi(\xi,\tau)=\exp\left\{-i\frac{\xi^2}{20\tau}\right\}v(\xi,\tau)$ resulting in 
\begin{equation}
i\frac{\partial v}{\partial \tau}\,=\,-\, \frac{\partial^2 v}{\partial \xi^2}\,+\,\xi\,v\,
+\,\frac{1}{25}\,\frac{\xi^2}{\tau^2}v, \ \ \tau<0; \ \ \xi>0, \ \ v(0,\tau)=0, 
\label{PopW-}
\end{equation} 
for which the existence of a wave operator was established. 
We emphasise that the transformed problem \eqref{PopW-} recasts the original problem as a perturbation, when $\tau\to -\infty$, of 
related unperturbed problem with (self-adjoint) ``Airy'' operator $A^-_0v(\xi)=-v^{''}+\xi v$, $\xi>0$, $v(0)=0$,  having a discrete spectrum and eigenfunctions 
\eqref{EF-} associated with the incoming whispering gallery modes. 


In \cite{BS2}, V.M. Babich and the first author have proved  that the solution of \eqref{PopEq}--\eqref{PopAC} is in fact classical, using bootstrap-type techniques based on the asymptotic expansion \eqref{PopAsFull}, 
which has allowed to justify \eqref{PopAsFull} with 
proved error estimates. 
 Moreover, \cite{BS2} also proved that $\psi(x,t)$ is infinitely differentiable in $\overline{\Omega}$, and 
together with all its derivatives 
decays super-algebraically as $x\to +\infty$ uniformly at any bounded interval in $t$. 
Namely, for all non-negative 
integers $\alpha$, $\beta$ and $\gamma$, and for any real $A_1<A_2$, with $\partial$ denoting appropriate partial derivatives, 
\begin{equation}
\sup_{A_1\leq t\leq A_2}\left\vert x^\alpha\partial_t^\beta\partial_x^\gamma \psi(x,t)\right\vert\,\rightarrow\,0, \ \ \ 
  \mbox{ as } x\to +\infty. 
\label{UnifDec}
\end{equation}
Similar results and some generalisations were soon thereafter obtained by different methods in \cite{Nakam}. 
\vspace{.1in} 

The problem of fundamental interest is to determine asymptotic behaviour of the solution $\psi(x,t)$ 
of \eqref{PopEq}--\eqref{PopAC} as $t\to +\infty$, which would 
correspond to the scattered wave field past the inflection point. In \cite{BS1} a formal asymptotic solution matching 
to the incoming whispering gallery wave \eqref{PopAC}  was 
constructed in the shadow zone above the limit ray $x=t^3/6$, which curve corresponds to the straight line tangent to the 
physical boundary at the inflection point (Fig. \ref{fig1}). 
The constructed  asymptotics was observed to remain an asymptotic expansion as $t\to +\infty$ as long
as $x/t^3-1/6=:\mu \gg t^{-2}$, breaking down when $\mu\sim t^{-2}$. Following general boundary-layer 
recipes, it was suggested in \cite{BS1}, see formula (4.3) of \cite{BS1}, to seek the 
asymptotic expansion  to the solution $\psi(x,t)$ as $t\to +\infty$ in the vicinity of the limit ray, with respect to the stretched 
variable $\eta:=\mu t^2=x/t-t^2/6$, in the form
\begin{equation}
\psi(x,t)\,\sim\,t^{-1/2}\exp\left\{
i\frac{7}{120}t^5+\frac{i}{2}\eta t^3+\frac{i}{2}\eta^2 t
\right\}\,\sum_{m=0}^\infty t^{-m} G_m(\eta), \ \ 
\eta:= \frac{x}{t}-\frac{1}{6}t^2. 
\label{SLansatz}
\end{equation} 
The form \eqref{SLansatz} was interpreted in \cite{BS1} as a ``searchlight'' ansatz, for a beam anticipated to concentrate near 
the limit ray. 
The unknown functions $G_m(\eta)$ are required to have a prescribed (exponentially small) asymptotics as $\eta\to+\infty$, to match 
with the shadow asymptotics over the limit ray constructed in \cite{BS1}. 

An accompanying curious observation made in \cite{BS1} was that changing in \eqref{PopEq} the variables from $(x,t)$ to 
$(\eta,t)$ and introducing new unknown function $G(\eta,t)$ by 
\begin{equation}
\psi(x,t)\,=\,t^{-1/2}\exp\left\{
i\frac{7}{120}t^5+\frac{i}{2}\eta t^3+\frac{i}{2}\eta^2 t
\right\}\,  G(\eta,t), \ \ 
\eta= \frac{x}{t}-\frac{1}{6}t^2
\label{SLtransform}
\end{equation} 
results in the following simple  PDE for $G$, see formula (4.4) of \cite{BS1}: 
\begin{equation}
i\,t^2\frac{\partial G}{\partial t}\,+\,\frac{1}{2}\,\frac{\partial G}{\partial\eta^2}\,=\,0. 
\label{QCtransf}
\end{equation} 

Comparing \eqref{SLansatz} with \eqref{SLtransform}--\eqref{QCtransf}, one concludes that all one needs is to 
determine the main-order searchlight amplitudes $G_0(\eta)$ 
for 
\begin{equation}
G(\eta,t)\,\sim\,\sum_{m=0}^\infty t^{-m}G_m(\eta), \ \ \mbox{ as } t\to +\infty, 
\label{tayl1}
\end{equation}
 specifying all the higher-order amplitudes via 
recurrence relations 
\begin{equation}
G_{m+1}(\eta)\,=\,-\,\,\frac{i}{2m+2}{G_m^{''}(\eta)}, \ \ \ \ j=0,1,.... 
\label{tayl2}
\end{equation} 

Determination of $G_0$ however requires additional information about the unknown solution $\psi(x,t)$. 
In \cite{P86} Popov, using a Green's function representation for \eqref{PopEq}, has expressed searchlight amplitudes 
equivalent to $G_0(\eta)$ up to a change of relevant variables in terms of certain integrals 
involving still unknown solution's Neumann data 
$\partial_x\psi(0,t)=:f(t)$ at the boundary $x=0$. 
The formulas derived in \cite{P86} were implemented in \cite{PK} for computing the searchlight 
amplitudes numerically 
(where $f(t)$ was in turn computed via a finite difference method previously developed in \cite{PPsch}).

The derivation in \cite{P86} was made under an assumption 
that the surface ``flux'' $f(t)$ decays fast as $t\to +\infty$, more precisely 
\begin{equation}
\int_1^\infty |f(t)|\,\,t^p\, dt\,<\,+\infty, \ \ \ p=1,2,3,... \,\,.
\label{currentdecay}
\end{equation} 
The assumption \eqref{currentdecay}, yet unproven, appeared to be partially supported by numerical evidence, \cite{PPsch}. 
Still, it was proved thereafter by the first author in \cite{S91} that 
\begin{equation}
\int_1^\infty |f(t)|^2\,t^2\, dt\,<\,+\infty.
\label{S91est}
\end{equation} 

Some ideas of \cite{S91}, as well as those of \cite{BS1} and \cite{BS2}, serve as important ingredients for the present work. 
One of the main results of the present work is the following theorem on the searchlight amplitude $G_0$, which we 
state in a slightly more general form. 

\begin{theorem}\label{thm1}
Let $\psi(x,t)\in C^\infty(\overline{\Omega})$ be an infinitely smooth solution to \eqref{PopEq} satisfying boundary condition 
\eqref{PopBC} and 
decay condition \eqref{UnifDec}; in particular $\psi(x,t$ can be the solution to problem \eqref{PopEq}--\eqref{PopAC}. 

Then there exists (unique) $G_0(\eta)\in H^1(\mathbb{R})$ with $\eta G_0(\eta)\in L^2(\mathbb{R})$, and with the following 
properties. Let, for $t>0$, 
\begin{equation}
\psi^+_0(x,t)\,:=\,t^{-1/2}\exp\left\{
i\frac{7}{120}t^5+\frac{i}{2}\eta t^3+\frac{i}{2}\eta^2 t
\right\}\, G_0(\eta), \ \ \ \ \eta:=\,\frac{x}{t}-\frac{1}{6}t^2. 
\label{psi0}
\end{equation} 
Then 
\begin{equation}
\left\| \psi(\cdot,t)\,-\,\psi^+_0(\cdot,t)\right\|_{L^2(0,+\infty)}(t)\,\rightarrow\,0, \ \ \ 
\mbox{as }\,t\to +\infty. 
\label{psi0lim}
\end{equation} 
Moreover, for $G(\eta,t)$ introduced for $t>0$ by \eqref{SLtransform} when $\eta\geq -t^2/6$ and extended by zero for 
$\eta < -t^2/6$, as $t\to +\infty$, 
\begin{equation}
G(\eta,t)\,\rightharpoonup\,G_0(\eta)\,\, \mbox{weakly in } \,H^1(\mathbb{R}), \ \ 
\eta G(\eta,t)\,\rightharpoonup\,\eta G_0(\eta)\,\, \mbox{weakly in } \,L^2(\mathbb{R}). 
\label{gconv}
\end{equation} 
\end{theorem}

Theorem \ref{thm1}  provides a rigorous justification for the main-order term in the searchlight ansatz \eqref{SLansatz} 
formally constructed in \cite{BS2},  and establishes 
some properties of the amplitude $G_0(\eta)$ as well as some convergence properties as $t\to+\infty$. 

\begin{remark}
The transformation \eqref{SLtransform} leading to \eqref{QCtransf} and then to 
\eqref{tayl1} and \eqref{tayl2}, as 
discovered in \cite{BS1}, plays a key role for the purposes of the present work. 
Notice that a simple further change of variable $\tau=1/t$ 
immediately transforms \eqref{QCtransf} into a free Schr\"{o}dinger equation on $\mathbb{R}$: 
\begin{equation}
i\,\frac{\partial G}{\partial \tau}\,-\,\frac{1}{2}\,\frac{\partial^2 G}{\partial\eta^2}\,=\,0, \ \ \ \tau=t^{-1}>0. 
\label{QCtao}
\end{equation} 
This appears to be closely related to a 
``pseudoconformal symmetry'' of a free Schr\"{o}dinger equation, see e.g. \cite{Talanov,Sulem,Tao}. 
Indeed, as observed in \cite{OT12}, 
see also \cite{HOS19} p.92, 
an inner problem for the boundary inflection could be derived directly in (stretched) Cartesian 
rather than curvilinear coordinates. The former yields a potential-free Schr\"{o}dinger equation (i.e. as 
\eqref{PopEq} but with no $xt\psi$ term) but in a domain bounded by a cubic parabola. That can then be transformed 
to \eqref{PopEq} by ``straightening'' the boundary. Using that transformation ``in reverse'' in \eqref{PopEq}, i.e. introducing 
new variable $\zeta=x-\frac{1}{6}t^3$ and new function $\varphi$ by 
\begin{equation}
\psi(x,t)\,=\,\exp\left\{\frac{i}{2}\zeta t^2\,+\,i\frac{7}{120}t^5\right\}\varphi(\zeta,t)
\label{OckTransf}
\end{equation}
results in $\varphi$ solving 
\begin{equation} 
i\frac{\partial\varphi}{\partial t}+\frac{1}{2}
\frac{\partial^2\varphi}{\partial \zeta^2}\,=\,0,  
\label{pwe}
\end{equation}
which is a ``parabolic wave equation'' as in \cite{OT12}. 
Then the 
transformation of \eqref{pwe} leading to \eqref{QCtao} is: 
\begin{equation}
\eta\,=\,\frac{\zeta}{t}, \ \ \tau\,=\,\frac{1}{t}, \ \ \ 
\varphi(\zeta,t)\,=\,\tau^{1/2}\exp\left(i\frac{\eta^2}{2\tau}\right)G(\eta,\tau).
\label{PCtransf}
\end{equation} 
Therefore the free Schr\"{o}dinger equation is invariant under transformation \eqref{PCtransf} 
(with $\tau$ replaced by $-\tau$ or $G$ by its complex conjugate), cf. eg. \cite{Tao}.  
See also \cite{OT20} with further reference to \cite{Olv} for some further invariance properties of \eqref{pwe}. 

Combining the transformation \eqref{OckTransf} with \eqref{PCtransf} is easily seen to yield \eqref{SLtransform} with 
\eqref{QCtao}. 
These transformations imply in particular that \eqref{psi0}--\eqref{psi0lim} is held for the analogous problem 
on the whole line, $x\in \mathbb{R}$. A major technical challenge, resolved in the present work, 
is to prove that \eqref{psi0}--\eqref{psi0lim} is still held in the presence of the boundary condition \eqref{PopBC}. 
\hfill $\Box$  
\end{remark}

\begin{remark}\label{rem2}
Similarly to \eqref{PopW-}, the transformed problem \eqref{pwe} in domain $\{(\zeta,t): \zeta>-t^3/6\}$ can be viewed 
 as a perturbation, when $t\to +\infty$, of the 
unperturbed problem  on the whole of the $(\zeta,t)$-plane 
with related (self-adjoint) operator $A^+_0\varphi(\xi)=-\,\frac{1}{2}\varphi^{''}$, having a continuous spectrum associated with a 
``free scattering'' i.e. with no boundary. 
\end{remark}

A detailed proof of Theorem \ref{thm1} will be given in Section \ref{proofthm1}. 
In the next section we give some further interpretations of Theorem \ref{thm1}, in the light of Remark \ref{rem2}.

\section{Further interpretations of Theorem \ref{thm1}: searchlight wave operator and inflection scattering operator} 
\label{sect3}

In this section we give a less formal sketch, avoiding fine technical details, of how Theorem \ref{thm1} leads to further implications and 
interpretations in terms of a ``searchlight'' wave operator as $t\to+\infty$ and of an inflection scattering operator. 

Theorem \ref{thm1} applies for every incoming whispering gallery wave \eqref{PopAC}, i.e. for $j=1,2,3,...$. 
Notice that it also 
holds for any classical solution of \eqref{PopEq}--\eqref{PopBC} which is infinitely 
smooth, $\psi(x,t)\in C^\infty(\overline{\Omega})$, and decays 
as $x\to +\infty$ as in \eqref{UnifDec}. 
It immediately follows from 
\cite{BS2}, cf. Theorem 2.1 of \cite{BS2}, that the latter properties are satisfied 
 by 
solutions to a Cauchy problem to \eqref{PopEq}--\eqref{PopBC} with initial data from for example $C_0^\infty(0,+\infty)$ 
at $t_1=+1$:
\begin{equation}
\psi(x,t_1)\,=\,\Psi_1(x), \ \ \ \ \Psi_1\in C_0^\infty(0,+\infty). 
\label{IC}
\end{equation} 
This determines a linear mapping $U:C_0^\infty(0,+\infty)\,\to\,L^2(\mathbb{R})$, by $U\Psi_1:=G_0(\eta)$, where 
function $G_0$ is associated with the solution of \eqref{PopEq}, \eqref{PopBC}, \eqref{IC} according to Theorem \ref{thm1}. 

It is easy to see that $U$ is an $L^2$-isometry, namely 
\begin{equation}
\left\|G_0\right\|_{L^2(\mathbb{R})}\,=\,\left\|\Psi_1\right\|_{L^2(0,+\infty)}. 
\label{g0psiisom}
\end{equation} 
Indeed, denoting henceforth 
$\mathbb{R}^+:=(0,+\infty)$, and using the fact that a solution to \eqref{PopEq}--\eqref{PopBC} 
has a $t$-independent $L^2(\mathbb{R}^+)$-norm, given $\Psi_1\in C^\infty_0(\mathbb{R}^+)$, 
\begin{equation}
\|\psi(\cdot,t)\|_{L^2(\mathbb{R}^+)}\,=\,\|\Psi_1\|_{L^2(\mathbb{R}^+)}, \ \  \ \forall t\in\mathbb{R}. 
\label{psil2inv}
\end{equation} 
On the other hand, from \eqref{psi0}, 
\begin{eqnarray}
\|\psi^+_0(\cdot,t)\|^2_{L^2(\mathbb{R}^+)}\,=\,
\int_0^{+\infty}t^{-1}\left\vert G_0\left(\frac{x}{t}-\frac{1}{6}t^2\right)\right\vert^2dx\,\,=\,
\, \ \ \ \ \ \ \ \ \ \ \ \ \ \ \ \ \ \ \ \ \ 
\nonumber \\
\, \ \ \ \ \ \ \ \ \ \ \ \ \ \ \ \ \ \ \ \ \ 
\int_{-\frac{1}{6}t^2}^{+\infty}|G_0(\eta)|^2d\eta\,\,\longrightarrow\,\,\|G_0\|^2_{L^2(\mathbb{R})}, 
\ \ \ 
\mbox{as } \,t\to +\infty. 
\label{g0l2lim}
\end{eqnarray}
Hence, via  passing to the limit as $t\to +\infty$ and using \eqref{psi0lim} and \eqref{psil2inv}, 
\[
\|G_0\|_{L^2(\mathbb{R})}=\lim_{t\to +\infty}\|\psi^+_0(\cdot,t)\|_{L^2(\mathbb{R}^+)}=
\lim_{t\to +\infty}\|\psi-(\psi-\psi^+_0)(\cdot,t)\|_{L^2(\mathbb{R}^+)}= 
\|\Psi_1\|_{L^2(\mathbb{R}^+)}, 
\]
yielding \eqref{g0psiisom}. 

In a standard way, as $C_0^\infty(\mathbb{R}^+)$ is dense in $L^2(\mathbb{R}^+)$, the above isometry is 
uniquely extended by continuity from $C_0^\infty(\mathbb{R}^+)$ to $L^2(\mathbb{R}^+)$. 
Moreover, as immediately follows from the proof of Theorem \ref{thm1}, see Lemma \ref{lem3} below, 
the range of this extension is 
actually the whole of $L^2(\mathbb{R})$. Hence the so constructed $U$ is in fact a unitary operator, 
$U: L^2(\mathbb{R}^+)\to L^2(\mathbb{R})$. 

This all can be interpreted in terms of existence of a wave operator as $t\to +\infty$ as follows. 
Fixing $t_1>0$, e.g. $t_1=+1$, consider the Cauchy problem for \eqref{PopEq}--\eqref{PopBC} with initial 
condition \eqref{IC}. Its solution $\psi(x,t)$ determines for any 
$t\in\mathbb{R}$ an $L^2(\mathbb{R}^+)$-isometric ``propagator'' operator $U(t,t_1)$, by 
$\psi(\cdot,t)=U(t,t_1)\Psi_1$. Hence $U(t,t_1)$ is uniquely extended by density to a unitary 
operator in $L^2(\mathbb{R}^+)$, whose inverse $U^{-1}(t,t_1)=U(t_1,t)$ is similarly related to solution to  
the Cauchy problem with initial data at $t$. 

A physical interpretation of the searchlight ansatz \eqref{SLansatz}--\eqref{QCtransf} is that, for large positive 
$t$, the effect of the boundary $x=0$ becomes increasingly insignificant: the wave energy tends to 
concentrate in a neighbourhood of the limit ray $x=\frac{1}{6}t^3$ away from the boundary, so 
further propagation of the wave 
beyond the inflection point i.e. for $t\to +\infty$, 
to a good approximation, is such as if there were no boundary. 
This intuition suggests that for $t>0$ one can take as an unperturbed problem 
simply the problem for PDE \eqref{PopEq} on the whole line 
$x\in \mathbb{R}$ with a similar initial condition at $t=t_1$ as \eqref{IC} but for $x\in\mathbb{R}$. 
The latter problem is 
equivalently reformulated via transformation \eqref{SLtransform} 
as a Cauchy problem for \eqref{QCtransf} for $t>0$. 

Notice that transformation \eqref{SLtransform} leading further to \eqref{QCtao} similarly 
determines, for any $\tau,\tau_1>0$, a unitary propagator $\tilde U_0(\tau,\tau_1)$ in $L^2(\mathbb{R})$ 
for \eqref{QCtao}. Importantly the point $\tau=0$, corresponding to $t\to +\infty$, remains a regular point for 
\eqref{QCtao} hence with a well-defined (strong) limit of $\tilde U_0(\tau,\tau_1)$ as $\tau\to 0+$. 
Using now the transformation \eqref{SLtransform} in reverse, determines the unitary propagator $U_0(t,t_1)$ for the 
unperturbed (i.e. with no boundary) problem for \eqref{PopEq} with the above described asymptotic behaviour as $t\to +\infty$. 

All this can be interpreted in terms of existence of a unitary analogue of wave operator 
$W_+: L^2(\mathbb{R})\to L^2(\mathbb{R}^+)$ as $t\to+\infty$, which can be defined as 
\begin{equation}
W_+\,:=\,\mbox{s-}\lim_{t\to +\infty}U^{-1}(t,t_1)\chi_+U_0(t,t_1). 
\label{w+}
\end{equation}
Here $\chi_+$ denotes restriction of a function to positive semi-axis $\mathbb{R}^+$, 
and 
$\mbox{s-}\lim$ stands for a strong limit in $L^2(\mathbb{R})$. 
Notice the role of $\chi_+$ in \eqref{w+} for approximating for large positive $t$ the solutions to the problem with boundary 
($x=0$)  
in terms of solutions of the unperturbed problem without boundary, cf e.g. \eqref{g0l2lim2} in the proof of Lemma \ref{lem3} below. 

We state the above as the following theorem. 
\begin{theorem}\label{thm2}
There exists wave operator $W_+$ defined by \eqref{w+}, which is a unitary operator from $L^2(\mathbb{R})$ into 
$L^2(\mathbb{R}^+)$. 
\end{theorem}

As was shown in \cite{P82}, and in fact also follows  from \cite{BS2}, there exists also a 
``whispering gallery'' wave operator $W_-$ when $t\to -\infty$ for a reformulated Cauchy problem 
\eqref{PopW-} 
for $t<0$ 
in $L^2(\mathbb{R}^+)$. Taken together, this determines an analogue of a {\it scattering operator} for the inflection 
problem, via determining the searchlight amplitudes corresponding to each incoming whispering gallery wave 
\eqref{PopAC}, $j=1,2,...\,$. One way for such a construction is to notice that for any fixed $j$ and constant 
$D_j$ the $L^2(\mathbb{R}^+)$-norm of $\psi^-_0(x,t)$ given by \eqref{PopAC} is $t$-independent for $t<0$. 
Normalize $D_j$ in \eqref{PopAC} so that $\|\psi^-_0(\cdot,t)\|_{L^2(\mathbb{R}^+)}\equiv 1$ for all $t<0$, 
and let $G_0^{(j)}(\eta)$ be 
related searchlight amplitudes for the solution $\psi^{(j)}(x,t)$ as $t\to +\infty$ according to Theorem \ref{thm1}. 
Then, because of \eqref{PopAC0}, the $L^2(\mathbb{R}^+)$-norm $t$-independence of $\psi^{(j)}(x,t)$, and 
\eqref{g0psiisom}, $\|G_0^{(j)}\|_{L^2(\mathbb{R}^+)}=1$. 

Let $H_-=l^2$ be the standard Hilbert space of complex-valued square-summable sequences ${\bf a}=\left(a_j\right)_{j=1}^\infty$, 
\[
\left\|{\bf a}\right\|_{H_-}\,:=\,\left(\sum_{j=1}^{+\infty}|a_j|^2\right)^{1/2}. 
\]
Then it is natural to define a scattering operator as the map $S: l^2 \rightarrow L^2(\mathbb{R})$ by 
\[
S{\bf a}\,:=\,\sum_{j=1}^\infty a_j G_0^{(j)}(\eta). 
\]
All the above implies the following 
\begin{corollary}\label{cor1}
The above analogue of a scattering operator $S$ is a well-defined unitary operator from $l^2$ to $L^2(\mathbb{R})$. 
\end{corollary}
The physical meaning of the scattering operator $S$ is in specifying how the intensities of the incoming whispering gallery modes at 
$t\to -\infty$ 
transform into the searchlight amplitudes as $t\to +\infty$. 

\section{Proof of Theorem \ref{thm1}}\label{proofthm1}

It is more convenient to work with an equivalently transformed via \eqref{SLtransform} problem for $G(\eta,\tau)$ solving \eqref{QCtao}, which 
we re-write as 
\begin{equation}
i\, G_\tau\,-\,\frac{1}{2}\,G_{\eta\eta}\,=\,0, \ \ \ \tau=t^{-1}>0, 
\label{QCtao2}
\end{equation} 
using henceforth subscripts as shorthands for relevant partial derivatives. 
Equation \eqref{QCtao2} holds 
in domain 
$\tilde\Omega=\left\{(\eta,\tau):\, \tau>0,\,\, \eta>-\,\frac{1}{6}\tau^{-2}\right\}$ 
with boundary condition on a curve $l$ as follows 
\begin{equation}
G(\eta,\tau)\,=\,0, \ \ \ (\eta,\tau)\in l:=\left\{(\eta,\tau):\, \eta=-\,\frac{1}{6}\tau^{-2}, \,\tau>0\right\}. 
\label{BC2}
\end{equation} 
By \cite{BS2} $G\in C^\infty(\tilde\Omega\cup l)$, and the decay property \eqref{UnifDec} immediately translates into:  
for all non-negative 
integers $\alpha$, $\beta$ and $\gamma$, and for any $0<\delta<\tau_1<+\infty$, 
\begin{equation}
\sup_{\delta\leq \tau\leq \tau_1}\left\vert \eta^\alpha\partial_\tau^\beta\partial_\eta^\gamma 
G(\eta,\tau)\right\vert\,\rightarrow\,0, 
  \mbox{as } \eta\to +\infty. 
\label{UnifDec2}
\end{equation}
For proving \eqref{gconv}, and then \eqref{psi0lim}, we aim at establishing 
the limit behaviour of the above solution $G(\eta,\tau)$ as $\tau\to +0$. 
The following lemma establishes necessary a priori estimates for $G$ for $0<\tau<\tau_1$. 
Below, all the norms $\|\cdot\|$ are the $L^2$-norms for $\eta$ above the boundary curve $l$, i.e. in 
$L^2\left(-\frac{1}{6}\tau^{-2},+\infty\right)$. 
\begin{lemma}\label{lem1}
Let a smooth function $G(\eta,\tau)$ solve \eqref{QCtao2} in $\tilde\Omega$, so that 
$G\in C^\infty(\tilde\Omega\cup l)$ and \eqref{BC2} and \eqref{UnifDec2} hold, and let $\tau_1>0$. Then there exist 
positive constants $C_0$, $C_1$ and $C_2$, such that for all $0<\tau<\tau_1$, 
\begin{enumerate}
\item[(i)] 
\begin{equation}
\|G\|(\tau)\,\equiv\,C_0\,>\,0; 
\label{l1i}
\end{equation}
\item[(ii)] 
\begin{equation}
\int_0^{\tau_1}\tau^{-3}
\left\vert G_\eta\left(-\,\frac{1}{6}\tau^{-2},\tau\right)\right\vert^2d\tau\,<\,+\infty; 
\label{l1ii}
\end{equation}
\item[(iii)] 
\begin{equation}
\left\|G_\eta(\eta,\tau)\right\|(\tau)\,\leq\,C_1; 
\label{l1iii}
\end{equation}
\item[(iv)] 
\begin{equation}
\left\|\eta G(\eta,\tau)\right\|(\tau)\,\leq\,C_2; 
\label{l1iv}
\end{equation}
\end{enumerate}
\end{lemma}
\begin{proof}
Fix $\tau_1>0$, e.g. $\tau_1=1$. 

{\it (i):} \eqref{l1i} is standard. Multiplying \eqref{QCtao2} by $\overline{G}$ (with the overbar denoting the 
complex conjugate) and integrating in $\eta$ from $-\tau^{-2}/6$ to $+\infty$, integration by parts using 
\eqref{BC2} and \eqref{UnifDec2} and then taking the imaginary part results in 
$\frac{d}{d\tau}\left(\|G\|^2\right)=0$, which yields \eqref{l1i}. 

{\it (ii) \& (iii):} This follows a similar proof of \eqref{S91est} in \cite{S91}. 
Multiplying \eqref{QCtao2} by $\overline{G_{\eta\eta}}$ and taking the imaginary part results in 
$\overline{G_{\eta\eta}}G_\tau+{G_{\eta\eta}}\overline{G_\tau}=0$, or equivalently 
\[
\left(\overline{G_{\eta}}G_\tau+{G_{\eta}}\overline{G_\tau}\right)_\eta\,-\,
\left(\left\vert G_\eta\right\vert^2\right)_\tau\,=\,0. 
\]
Integrate the latter in $\eta$ from $-\tau^{-2}/6$ to $+\infty$ using \eqref{UnifDec2}, and notice that 
\[
\left(\left\|G_\eta\right\|^2\right)_\tau\,=\,
\frac{d}{d\tau}\left(\int_{-\tau^{-2}/6}^{+\infty}\left|G_\eta(\eta,\tau)\right|^2d\eta\right)= 
\int_{-\tau^{-2}/6}^{+\infty}\left(\left|G_\eta \right|^2\right)_\tau d\eta 
-\frac{1}{3}\tau^{-3}\left(\left\vert G_\eta\right\vert^2\right)_l, 
\]
where $(\cdot)_l$ denotes the value of the expression in the brackets on the curve $l$ i.e. for 
$\eta=-\tau^{-2}/6$. 
As a result, 
\begin{equation}
-\,\left(\overline{G_{\eta}}G_\tau+{G_{\eta}}\overline{G_\tau}\right)_l\,-\,
\left(\left\|G_\eta\right\|^2\right)_\tau
-\frac{1}{3}\tau^{-3}\left(\left\vert G_\eta\right\vert^2\right)_l\,=\,0.  
\label{l1ii2}
\end{equation} 
Now as $(G)_l= 0$ for all $\tau$, the tangential derivative of $G$ along $l$ is zero, 
implying $G_\tau+\frac{1}{3}\tau^{-3}G_\eta=0$ on $l$ or 
$G_\tau=-\frac{1}{3}\tau^{-3}G_\eta$. 
Using the latter in the first term of \eqref{l1ii2} results in 
$\frac{1}{3}\tau^{-3}\left(\overline G_\eta G_\eta+G_\eta\overline G_\eta\right)_l-
\left(\left\|G_\eta\right\|^2\right)_\tau-\frac{1}{3}\tau^{-3}\left(\left\vert G_\eta\right\vert^2\right)_l\,=\,0$ or 
\begin{equation}
\tau^{-3} \left(\left\vert G_\eta\right\vert^2\right)_l\,=\,
{3}
\left(\left\|G_\eta\right\|^2\right)_\tau. 
\label{l1ii3}
\end{equation}
Denoting $g(\tau):=\left( G_\eta \right)_l=G_\eta(-\tau^{-2}/6,\tau)$, and integrating 
\eqref{l1ii3} from some $\tau>0$ to $\tau_1$, $\tau<\tau_1$, results in 
\begin{equation}
\int_\tau^{\tau_1}{\tau'}^{-3}|g(\tau')|^2d\tau'\,+\,
{3}\left\|G_\eta\right\|^2(\tau)\,=\,{3} \left\|G_\eta\right\|^2(\tau_1). 
\label{l1ii4}
\end{equation} 
For fixed $\tau_1>0$, the right hand side of \eqref{l1ii4} is constant, which immediately yields 
\eqref{l1iii} for $0<\tau<\tau_1$. 
For \eqref{l1ii}, it only remains to pass to the limit in \eqref{l1ii4} as 
$\tau\to +0$. 

{\it (iv):}
Multiply now  \eqref{QCtao2} by $\eta^2\overline{G}$ and take the imaginary part, giving 
\[
\eta^2\left(\overline{G}G_\tau+{G}\overline{G_\tau}\right)
-\frac{1}{2i}\eta^2\left(\overline{G}G_{\eta\eta}-{G}\overline{G_{\eta\eta}}\right)\,=\,0, 
\]
which can be equivalently rewritten as 
\begin{equation}
\left(\eta^2|G|^2\right)_\tau\,+\,
\frac{i}{2}\left[
\eta^2\left(\overline{G}G_{\eta}-{G}\overline{G_\eta}\right)\right]_\eta\,-\,
i\eta\left(\overline{G}G_{\eta}-{G}\overline{G_\eta}\right)\,=\,0. 
\label{l1iv2}
\end{equation}
Integrating \eqref{l1iv2} over $\eta$, the middle term vanished due to \eqref{BC2} and \eqref{UnifDec2}, and as 
a result
\begin{equation}
\left(\left\|\eta G\right\|^2\right)_\tau\,=\,i\int_{-\frac{1}{6}\tau^{-2}}^{+\infty}
\left(\eta\overline{G}G_{\eta}\,-\,\eta{G}\overline{G_\eta}\right)d\eta.
\label{l1iv3}
\end{equation} 
Denoting, for $0<\tau\leq \tau_1$, 
$I(\tau):=\left\|\eta G\right\|(\tau)=
\left(\int_{-\frac{1}{6}\tau^{-2}}^{+\infty}\eta^2|G(\eta,\tau)|^2d\eta\right)^{1/2}$, and applying 
Cauchy-Schwartz inequality to the right hand side of \eqref{l1iv3} results in 
$|2I I_\tau|\leq\,2I\|G_\eta\|$. As a result, via \eqref{l1iii}, $|I_\tau(\tau)|\leq\|G_\eta\|(\tau)\leq C_1$, and hence 
\[
\|\eta G\|(\tau)=I(\tau)\,\leq\, I(\tau_1)+C_1\tau_1\,=:\,C_2,
\]
yielding \eqref{l1iv}. 
\end{proof}

We regard, for $0<\tau\leq \tau_1$, $G(\cdot,\tau)$ as defined on the whole of the real line $\mathbb{R}$ by extending 
it by zero for $\eta<-\tau^{-2}/6$. 
Lemma \ref{lem1} provides boundedness of $G(\eta,\tau)$ in the $H^1(\mathbb{R})$ norm (via \eqref{l1i} and \eqref{l1iii}), 
and of $\eta G(\eta,\tau)$ in the $L^2(\mathbb{R})$ norm, via \eqref{l1iv}. 
The following well-known lemma establishes analogues of compactness results under these conditions. 
As it plays an important role for our proof of Theorem \ref{thm1} we give its short proof for the readers' convenience. 
\begin{lemma}\label{lem2}
Let, for all $0<\tau<\tau_1$, $G(\eta,\tau)\in H^1(\mathbb{R})$ with $\eta G(\eta,\tau)\in L^2(\mathbb{R})$ be 
such that that the respective $H^1$ and $L^2$ norms are bounded, i.e. \eqref{l1i}, \eqref{l1iii} and \eqref{l1iv} hold. 
Then there exists a subsequence $\tau_n\to +0$ and $G_0\in H^1(\mathbb{R})$ with $\eta G_0(\eta)\in L^2(\mathbb{R})$, such that 

\item[(i)]
\begin{equation}
G(\eta,\tau_n)\,\rightharpoonup\,G_0(\eta)\, \mbox{weakly in } \,H^1(\mathbb{R}), \ 
\eta G(\eta,\tau_n)\,\rightharpoonup\,\eta G_0(\eta)\, \mbox{weakly in } L^2(\mathbb{R}). 
\label{gconv1}
\end{equation} 

\item[(ii)] 
\begin{equation}
G(\eta,\tau_n)\,\rightarrow\,G_0(\eta)\,\, \mbox{strongly in } \,L^2(\mathbb{R}). 
\label{gconv2}
\end{equation}
\end{lemma}  
\begin{proof}
(i) By the lemma's assumptions, for all $0<\tau<\tau_1$, $G(\cdot,\tau)$ belongs to Hilbert space $H$ of functions from $H^1(\mathbb{R})$ such that $\|u\|^2_{H}:=\|u\|^2_{H^1(\mathbb{R})}+\|\eta u\|^2_{L^2(\mathbb{R})}<\infty$, which specifies the norm in $H$. Then it is easy to see that the theorem on weak compactness of a unit ball implies the existence of $G_0\in H$ and of a subsequence $\tau_n\to 0+$, such that (43) holds.

(ii)
For a subsequence chosen in (i), 
for any $N=1,2,...$, 
\eqref{gconv1} and the Rellich compactness theorem ensure that 
\begin{equation}
G(\eta,\tau_n)\,\rightarrow\,G_0(\eta)\,\,\mbox{ strongly in } \,L^2(-N,N). 
\label{l2ii1}
\end{equation}
Now, denoting here for brevity $G_n(\eta):=G(\eta,\tau_n)$, 
\[
\|G_n-G_0\|_{L^2(\mathbb{R})}^2\,=\,\int_{-N}^N|G_n-G_0|^2d\eta+ \int_{\eta>N}|G_n-G_0|^2d\eta\,\,\leq 
\]
\begin{equation}
\int_{-N}^N|G_n-G_0|^2d\eta+ N^{-2}\int_{|\eta|>N}\eta^2|G_n-G_0|^2d\eta\leq 
\int_{-N}^N|G_n-G_0|^2d\eta\,+\,CN^{-2},\, \forall n, N, 
\label{l2ii2}
\end{equation} 
with a constant $C$ independent of $n$ and $N$. (With latter inequality following from \eqref{l1iv} 
and (i) .) 
Now taking limit-supremum as $n\to\infty$ of inequality \eqref{l2ii2} for every $N$ and using \eqref{l2ii1} 
results in 
$\limsup_{n\to\infty}\|G_n-G_0\|^2\leq CN^{-2}$ for all $N\in\mathbb{N}$. 
Finally, \eqref{gconv2} follows from the latter as $N$ is arbitrary. 
\end{proof} 
To complete the proof of Theorem \ref{thm1} we essentially need to show that for the actual solution 
satisfying \eqref{QCtao2}--\eqref{UnifDec2} the choice of $G_0$ according to Lemma \ref{lem2} is 
actually independent of the subsequence and therefore the convergences in \eqref{gconv1} and 
\eqref{gconv2} actually hold for the whole of $t\to +0$. 
The latter is being achieved below by adapting and refining certain technical arguments from 
\cite{BS2}, as follows. 

To this end, we first aim at showing that, given a well-behaved function $G_0(\eta)$, there exists a well-behaved 
solution to \eqref{QCtao2}--\eqref{BC2} for $\tau>0$, such that $G_0$ is its strong $L^2$-limit as 
$\tau\to +0$. To begin with, we can take as such $G_0\in C_0^\infty(\mathbb{R})$. 
The following key lemma holds: 
\begin{lemma}\label{lem3}
Let $G_0\in C_0^\infty(\mathbb{R})$. Then there exists unique $G(\eta,\tau)\in C^\infty(\tilde\Omega\cup l)$ a solution to 
\eqref{QCtao2} in $\tilde\Omega$, satisfying \eqref{BC2} and \eqref{UnifDec2}, such that 
\begin{equation}
\|G(\cdot,\tau)-G_0\|\,\to\,0 \ \ \ \mbox{ as } \tau\to +0. 
\label{lem3conv}
\end{equation} 
Equivalently, there exists a unique solution $\psi(x,t)\in C^\infty(\overline{\Omega})$ to \eqref{PopEq}, 
satisfying \eqref{PopBC}, \eqref{UnifDec}, such that \eqref{psi0lim} holds for $\psi^+_0(x,t)$ defined by \eqref{psi0}. 
\end{lemma}
\begin{proof}
The proof follows the ideas from Lemmas 3.7--3.9 and Theorem 3.1 of \cite{BS2}, readjusting the argument as necessary from 
the asymptotic condition \eqref{PopAC0} for $t\to -\infty$ to \eqref{psi0lim} as $t\to +\infty$. 
In this respect, it is more convenient to deal with the latter statement of the lemma, with the stated equivalence 
following via the transformation \eqref{SLtransform}, \eqref{QCtao}. 
Indeed, let $G_0(\eta)\in C_0^\infty(\mathbb{R})$ with associated $G(\eta,\tau)$, and related $\psi_0^+(x,t)$ and $\psi(x,t)$ 
determined by respectively \eqref{psi0} and \eqref{SLtransform} (with $\tau=t^{-1}$ in the latter). 
Then, for the left hand side of \eqref{psi0lim}, for large enough $t$, 
\begin{eqnarray}
\|\psi(x,t)-\psi^+_0(x,t)\|^2_{L^2(\mathbb{R}^+)}\,\,=\,\, \ \ \ \ \ \ \ \ \ \ \ \ \ \ \ \ \ \ \ 
\nonumber\\
\int_0^{+\infty}t^{-1}\left\vert G\left(\frac{x}{t}-\frac{1}{6}t^2,t\right)-
G_0\left(\frac{x}{t}-\frac{1}{6}t^2\right)\right\vert^2dx\,= 
\nonumber \\
\int_{-\frac{1}{6}\tau^{-2}}^{+\infty}\left|G(\eta,\tau)-G_0(\eta)\right|^2d\eta\,=\,
\int_{-\infty}^{+\infty}\left|G(\eta,\tau)-G_0(\eta)\right|^2d\eta\,=
\nonumber \\ 
\,\,\,\, \ \  \ \ \ \ \ \ \ \ 
\left\|G(\cdot,\tau)\,-\,G_0\right\|^2_{L^2(\mathbb{R})}. 
\label{g0l2lim2}
\end{eqnarray}
(Since the support of $G_0(\eta)$ is finite and $G(\eta,\tau)$ is defined to vanish for $\eta<-\frac{1}{6}\tau^{-2}$, 
for small enough $\tau$ the second integral in \eqref{g0l2lim2} could be replaced by that over the whole of the real line.) 
Thus \eqref{g0l2lim2} establishes the stated equivalence.

1. Let now $G_0\in C_0^\infty(\mathbb{R})$ have a support in $[-\Lambda,\Lambda]$ for some $\Lambda>0$. Choose some large 
enough $N\in \mathbb{N}$ and define an asymptotic solution to \eqref{PopEq} as $t\to+\infty$, cf \eqref{SLansatz}, 
\begin{equation}
\psi^{(N)}(x,t)\,:=\,t^{-1/2}\exp\left\{
i\frac{7}{120}t^5+\frac{i}{2}\eta t^3+\frac{i}{2}\eta^2 t
\right\}\,\sum_{m=0}^N t^{-m} G_m(\eta), \ \ 
\eta= \frac{x}{t}-\frac{1}{6}t^2,  
\label{SLansatz2}
\end{equation} 
where $G_m$, $m=1,2,...,N$, are determined from $G_0$ by \eqref{tayl2}. 
It then follows that $\psi^{(N)}(x,t)\in C^\infty (\overline\Omega_*)$ 
and the support of $\psi^{(N)}$ in $\Omega_*$ is separated from the boundary $x=0$ of $\Omega$, where 
$\Omega_*:=\{(x,t):\,x>0, \ t>(6\Lambda)^{1/2}+1=:t_*\}$. 
In particular, 
as $\psi^{(N)}(x,t)$ vanishes for $\frac{x}{t}-\frac{1}{6}t^2<-\Lambda$, i.e. for 
$x<\frac{1}{6}t(t^2-6\Lambda)$, 
$\psi^{(N)}$ satisfies the boundary condition \eqref{PopBC} in $\Omega_*$. 
Moreover, 
as $\psi^{(N)}$ also vanishes for  $\frac{x}{t}-\frac{1}{6}t^2>\Lambda$, 
the decay condition \eqref{UnifDec} is trivially satisfied by $\psi^{(N)}$ for any $t_*<A_1<A_2<+\infty$. 

We aim at constructing solutions to \eqref{PopEq}--\eqref{PopBC} with Cauchy data given by \eqref{SLansatz2} 
at large enough $t=T$ and ultimately showing that these solutions converge to the sought solution 
$\psi(x,t)$ as $T\to+\infty$. 
So, following p.1570 of \cite{BS2}, for all $\xi>t_*$ we define $\psi^{[N]}(x,t;\xi)$ as solution to the  
Cauchy problem for \eqref{PopEq}--\eqref{PopBC} in $\Omega_*$ with initial data at $t=\xi$: 
\begin{equation}
L\psi^{[N]}:=\,i\psi^{[N]}_t\,+\,\frac{1}{2}\psi^{[N]}_{xx}\,+\,xt\,\psi^{[N]}\,=\,0 \ \ \ \mbox{in }\,\Omega_*, 
\label{psiNeq}
\end{equation} 
\begin{equation}
\psi^{[N]}(0,t;\xi)\,=\,0, \ \ \ 
\psi^{[N]}(x,\xi;\xi)\,=\,\psi^{(N)}(x,\xi). 
\label{psiNCP}
\end{equation}
By Theorem 2.1 of \cite{BS2} and its proof, for any $\xi>t_*$, there exists such a unique function 
$\psi^{[N]}(x,t;\xi)\in C^\infty(\overline{\Omega_*})$ satisfying \eqref{UnifDec} for any $t_*< A_1<A_2<+\infty$.

2. Following \cite{BS2} further, we next argue that for any $(x,t)\in\Omega_*$ and $\xi>t_*$, $\psi^{[N]}(x,t;\xi)$ is 
continuously differentiable in $\xi$, 
\begin{equation}
\frac{\partial\psi^{[N]}(x,t;\xi)}{\partial\xi}\,=:\,V_N(x,t;\xi),
\label{dpsivn}
\end{equation} 
and the derivative $V_N(x,t;\xi)$ is  a unique solution in $\Omega_*$ to  
another Cauchy problem
\begin{equation}
LV_N\,=\,0, \ \ \ 
V_N(0,t;\xi)\,=\,0, \ \ \ 
V_N(x,\xi;\xi)\,=\,-iL\psi^{(N)}. 
\label{cauchy2}
\end{equation}
Again by Theorem 2.1 of \cite{BS2}, as defined by \eqref{cauchy2}, 
$V_N \in C^\infty(\overline{\Omega_*})$ and satisfies \eqref{UnifDec} for any $t_*<A_1<A_2<+\infty$.

The relation \eqref{dpsivn} between the solutions of 
\eqref{psiNeq}--\eqref{psiNCP} and \eqref{cauchy2} 
 is proved in Lemma 3.7 of \cite{BS2}.\footnote{Formally, 
\eqref{dpsivn}--\eqref{cauchy2} can be shown  
by differentiating the Cauchy data in \eqref{psiNCP} with respect to $\xi$: 
$\psi^{[N]}_t(x,t;\xi)+\psi^{[N]}_\xi(x,t;\xi)=\psi^{(N)}_\xi(x,\xi)$ when $t=\xi$, and then as $\psi^{[N]}$ solves 
\eqref{psiNeq}, 
\[
\psi^{[N]}_t(x,t;\xi)=\frac{i}{2}\psi^{[N]}_{xx}(x,t;\xi)+ixt\psi^{[N]}(x,t;\xi)= 
\frac{i}{2}\psi^{(N)}_{xx}(x,\xi)+ix\xi\psi^{(N)}(x,\xi), \ \ t=\xi 
\]
(having in the last equality used \eqref{psiNCP} again), which leads to \eqref{cauchy2}. 
This derivation has however to be adjusted as the differentiability of $\psi^{[N]}$ with respect to the parameter $\xi$ is not known 
in advance, see proof of Lemma 3.7 of \cite{BS2} for the details. 
}
 Therefore, for any 
$t_*<T_1<T_2<+\infty$ and for any $(x,t)\in\Omega_*$, 
\begin{equation}
\psi^{[N]}(x,t; T_2)\,-\,\psi^{[N]}(x,t; T_1)\,=\,\int_{T_1}^{T_2}V_N(x,t;\xi)\,d\xi.  
\label{ftc}
\end{equation}
Further following \cite{BS2}, given a function $\psi(x,t)\in C^\infty(\overline\Omega)$ obeying \eqref{UnifDec}, for any non-negative integers 
$\alpha$, $\beta$ and $\gamma$, we introduce semi-norms 
\[
\|\psi\|_{\alpha\beta\gamma}(t)\,:=\,\left(\int_0^{+\infty} 
x^{2\alpha}\left\vert \partial_t^\beta\partial_x^\gamma \psi(x,t)\right\vert^2dx
\right)^{1/2}. 
\]
The technical argument in the proof of Lemma 3.7 of \cite{BS2} 
ensures that  
\eqref{dpsivn} is infinitely differentiable in $x$ and $t$  with 
$\partial^\beta_t\partial^\gamma_xV_N(x,t;\xi)$ continuous in $\xi$ for any $\beta$ and $\gamma$. 
As a result, the identity \eqref{ftc} is also infinitely differentiable in $x$ and $t$, and 
yields Cook's type estimates, cf. e.g. \cite{Kato}
\begin{equation}
\left\|\psi^{[N]}(x,t; T_2)\,-\,\psi^{[N]}(x,t; T_1)\right\|_{\alpha\beta\gamma}(t)\,\leq\,
\int_{T_1}^{T_2}\left\|V_N(x,t;\xi)\right\|_{\alpha\beta\gamma}(t)\,d\xi.  
\label{cook}
\end{equation}

3. As $\psi^{(N)}$ is an asymptotic solution to \eqref{PopEq} as $t\to+\infty$, it directly follows 
from its construction, see \eqref{SLansatz}, \eqref{QCtransf}--\eqref{tayl2}, that for $t>t_*$ 
\begin{equation}
\left\|L\psi^{(N)}\right\|_{\alpha\beta\gamma}(t)\,\leq\,C(\alpha,\beta,\gamma,N)\,
t^{-2+3\alpha+4\beta+2\gamma-N}
\label{psiNest1}
\end{equation} 
with $t$-independent constants $C(\alpha,\beta,\gamma,N)$. 

Next, as problem \eqref{cauchy2} has initial data satisfying estimates \eqref{psiNest1} 
at $t=\xi$, 
\cite{BS2} 
implies similar estimate for $V_N$, cf Lemma 3.8 of \cite{BS2}. 
Namely, generally for a function $\psi(x,t)$ solving \eqref{PopEq}--\eqref{PopBC} and satisfying 
\eqref{UnifDec}, according to inequality (3.1) of \cite{BS2}, for any $t_1,t_2\in\mathbb{R}$, 
\begin{equation}
\|\psi\|_{\alpha\beta\gamma}(t_1)\,\leq\,\sum_{|J'|\leq|J|, \beta'=0}C_{JJ'} 
\left(
|t_1|^{|J|-|J'|}+|t_2|^{|J|-|J'|}+1
\right)\|\psi\|_{\alpha'0\gamma'}(t_2). 
\label{3.1BS2}
\end{equation}
Here for multiindices $J=(\alpha,\beta,\gamma)$ and $J'=(\alpha',\beta',\gamma)$, we denote 
$|J|:=3\alpha+4\beta+2\gamma$ and $|J'|:=3\alpha'+4\beta'+2\gamma'$. The constants 
$C_{JJ'}$ are independent of $t_1$ and $t_2$. 

Applying \eqref{3.1BS2} to $\psi(x,t)=V_N(x,t;\xi)$ with $t_1=t$ and $t_2=\xi$, 
for $t_*<A\leq t\leq B <+\infty$ and $\xi>t_*$, we obtain via \eqref{cauchy2} and \eqref{psiNest1} 
\begin{equation}
\left\|V_N(x,t;\xi)\right\|_{\alpha\beta\gamma}(t)\,\leq\,C(\alpha,\beta,\gamma,A,B,N)\,
\xi^{-2+3\alpha+4\beta+2\gamma-N}
\label{VNest1}
\end{equation} 
with constant $C(\alpha,\beta,\gamma,A,B,N)$ independent from $\xi$ and $t$. 

Given $N$ and  $1\leq t_*<A<B<+\infty$, \eqref{cook} and \eqref{VNest1} imply that, 
for all $\alpha$, $\beta$ and $\gamma$ with  $|J|\leq N$,
 $A\leq t\leq B$, and $t_*<T_1<T_2$, 
\[
\left\|\psi^{[N]}(x,t; T_2)-\psi^{[N]}(x,t; T_1)\right\|_{\alpha\beta\gamma}(t)\,\leq\,
C(\alpha,\beta,\gamma,A,B,N)
\int_{T_1}^{T_2}\xi^{-2}d\xi\,\longrightarrow\, 0, 
\]
as $T_1,T_2\to +\infty$.  
Therefore, as $A$ and $B$ such as $t_*<A<B<+\infty$ are arbitrary, there exists a function $\psi^{[N]}(x,t)$ such that 
$x^\alpha\partial^\beta_t\partial^\gamma_x\psi^{[N]}\in L^2(\mathbb{R}^+)$ 
for all $\alpha$, $\beta$ and $\gamma$ with 
$|J|\leq N$ and for all $t>t_*$, and 
\begin{equation}
\left\|\psi^{[N]}(x,t; T)\,-\,\psi^{[N]}(x,t)\right\|_{\alpha\beta\gamma}(t)\,\,
\rightarrow\,0 \ \ \mbox{as }\,T\, \to +\infty.  
\label{psiTest}
\end{equation} 

4. The rest of the proof follows proof of Lemma 3.9 of \cite{BS2} and shows that for large enough 
$N$ the above constructed $\psi^{[N]}(x,t)$ does not actually depend on $N$ and is the sought unique 
smooth solution $\psi(x,t)$. 

First, \eqref{psiTest} implies via the embedding theorems that, for large enough $N$, $\psi^{[N]}(x,t)$ is a classical solution 
of \eqref{PopEq} satisfying boundary condition \eqref{PopBC}. Further, since a solution to \eqref{PopEq}--\eqref{PopBC} has 
a $t$-independent $L^2(\mathbb{R}^+)$ norm, the uniqueness would immediately follow from \eqref{psi0lim}. So all what 
remains to show is that, for large enough $N$, $\psi^{[N]}(x,t)$ obeys \eqref{psi0lim}, i.e. 
(with $\|\cdot\|$ denoting the $L^2(\mathbb{R}^+)$ norm) 
\begin{equation}
\left\| \psi^{[N]}(x,t)\,-\,\psi^+_0(x,t)\right\|(t)\,\rightarrow\,0, \ \ \ 
\mbox{as }\,t\to +\infty. 
\label{keyconv}
\end{equation} 
We have, for all $t,T>t_*$,  
\begin{eqnarray}
\left\| \psi^{[N]}(x,t)\,-\,\psi^+_0(x,t)\right\|\,\,\leq\,\, 
\left\| \psi^{[N]}(x,t)-\psi^{[N]}(x,t;T)\right\|\,\,+ \ \ \ \ \ \ \ \ \ \ \ \ \ \ \ \ \ \ \ \ \ 
\nonumber \\
\, \ \ \ \ \ \ \ \ \ \ \ \ \ \ \ \ \ \ \ \ \ 
\left\| \psi^{[N]}(x,t;T)-\psi^{(N)}(x,t)\right\|\,+\,
\left\| \psi^{(N)}(x,t)-\psi^+_0(x,t)\right\|. 
\label{psiNlim}
\end{eqnarray}
Keeping first $t>t_*$ fixed and passing to the limit as $T\to+\infty$, 
the first term on the right hand side vanishes due to \eqref{psiTest}. 
Next, as the second term vanishes for $t=T$ by \eqref{psiNCP} and $L\psi^{[N]}=0$, due to a 
general estimate $\left\vert \frac{d}{dt}\|\varphi\|\,\right\vert\leq \|L\varphi\|$, 
for $T>t$, 
\[
\left\| \psi^{[N]}(x,t;T)-\psi^{(N)}(x,t)\right\|\,\leq\,
\left\vert\int_t^T\left\|L\psi^{(N)}(x,t')\right\|dt'\right\vert\,\,\,\, \leq\, 
\ \ \ \ \ \ \ \ \ \ \ \ \ \ \ \ \ \ \ \ \ 
\]
\begin{equation}
\ \ \ \ \ \ \ \ \ \ \ \ \ \ \ \ \ \ \ \ \ 
\int_t^{+\infty}\left\|L\psi^{(N)}(x,t')\right\|dt'\,\,\leq\,\,
C_N t^{-1-N} 
\label{2ndterm}
\end{equation} 
with some $t$-independent constant $C_N$. (The last inequality has used \eqref{psiNest1} 
with $\alpha=\beta=\gamma=0$.)  
Passing finally to the limit as $t\to+\infty$, the second term vanishes by \eqref{2ndterm}. 
Finally, the third term on the right hand side of \eqref{psiNlim} vanishes as $t\to+\infty$ as 
the $L^2$-norm of every ``higher-order'' term in \eqref{SLansatz2}, i.e. for $m\geq 1$, vanishes in the limit. 

This proves \eqref{keyconv}, and for completing the proof of the lemma it remains to observe that by the uniqueness 
$\psi(x,t):=\psi^{[N]}(x,t)$ is, for large enough $N$, independent of $N$. 
Therefore, taking $N$ large enough, 
$x^\alpha\partial^\beta_t\partial^\gamma_x\psi=x^\alpha\partial^\beta_t\partial^\gamma_x\psi^{[N]}\in L^2(\mathbb{R}^+)$ 
for all non-negative 
$\alpha$, $\beta$ and $\gamma$, and so $\psi(x,t)\in C^\infty(\overline{\Omega_*})$ with \eqref{UnifDec} held. 
As a solution to \eqref{PopEq}, \eqref{PopBC}, by e.g. Theorem 2.1 of \cite{BS2}, it can be uniquely continued for 
all $t\in\mathbb{R}$ with all the required properties, so the lemma is proved. 
\end{proof}

Lemma \ref{lem3} allows to complete the proof of Theorem \ref{thm1} by implying that the $G_0$ constructed in Lemma \ref{lem2} 
does not actually depend on the choice of the subsequence $\tau_n$ and the stated convergences therefore hold for the whole 
of $\tau\to+ 0$, as follows. 

\begin{lemma}\label{lem4}
Let $\psi(x,t)\in C^\infty(\overline{\Omega})$ be a solution to \eqref{PopEq}, satisfying the boundary condition \eqref{PopBC} 
and the decay condition \eqref{UnifDec}, and let function $G(\eta,\tau)$ be associated with it via 
the transformation \eqref{SLtransform}, \eqref{QCtao}. 

Then the convergences in \eqref{gconv1} and \eqref{gconv2} are held for the whole of $\tau\to+ 0$, and such $G_0$ is therefore 
unique. 

For this $G_0$ assertions \eqref{psi0lim} and \eqref{gconv} are valid, and therefore Theorem \ref{thm1} holds. 
\end{lemma}
\begin{proof}
Let $\psi(x,t)$ be a solution to \eqref{PopEq} satisfying \eqref{PopBC} 
and \eqref{UnifDec}, and let $G(\eta,\tau)$ be associated with it via \eqref{SLtransform}, \eqref{QCtao}. 
Then $G(\eta,\tau)$ satisfies assumptions of Lemma \ref{lem1} and therefore of Lemma \ref{lem2}. 
Applying to $G(\eta,\tau)$ Lemma \ref{lem2}, let $\tau_n\to+0$ be a related subsequence, and let $G_0\in H^1(\mathbb{R})$ 
with $\eta G_0(\eta)\in L^2(\mathbb{R})$ be the associated limit, see \eqref{gconv1} and \eqref{gconv2}. 
Using the transformation \eqref{SLtransform}, \eqref{QCtao} in reverse, we conclude that for $t_n=\tau_n^{-1}\to+\infty$ as 
$n\to\infty$, 
\begin{equation}
\left\| \psi(x,t_n)\,-\,\psi^+_0(x,t_n)\right\|_{L^2(\mathbb{R}^+)}(t_n)\,\rightarrow\,0, \ \ \ 
\mbox{as }\,n\to +\infty, 
\label{psi0limn}
\end{equation} 
where $\psi_0^+$ is given by \eqref{psi0}. 
This means that \eqref{psi0lim} is held for the subsequence 
$t_n\to+\infty$, and we need to show that it actually holds for the whole of $t\to+\infty$. 

Since $C_0^\infty(\mathbb{R})$ is dense in $L^2(\mathbb{R})$, for any $\varepsilon>0$ choose 
$G_0^\varepsilon\in C_0^\infty(\mathbb{R})$ such that 
$\|G_0^\varepsilon-G_0\|_{L^2(\mathbb R)}<\varepsilon$. 
Applying for every such $G_0^\varepsilon$ Lemma \ref{lem3}, let $\psi^\varepsilon(x,t)\in C^\infty(\overline{\Omega})$ be 
associated solution to \eqref{PopEq}, satisfying \eqref{PopBC} and \eqref{UnifDec}, so that \eqref{psi0lim} and 
\eqref{psi0} read
\begin{equation}
\left\| \psi^\varepsilon(x,t)\,-\,\psi^\varepsilon_0(x,t)\right\|_{L^2(\mathbb{R}^+)}(t)\,\rightarrow\,0, \ \ \ 
\mbox{as }\,t\to +\infty, 
\label{psi0limeps}
\end{equation} 
where 
\begin{equation}
\psi^\varepsilon_0(x,t)\,=\,t^{-1/2}\exp\left\{
i\frac{7}{120}t^5+\frac{i}{2}\eta t^3+\frac{i}{2}\eta^2 t
\right\}\, G^\varepsilon_0(\eta), \ \ \ \ \eta:=\,\frac{x}{t}-\frac{1}{6}t^2. 
\label{psi0eps}
\end{equation} 
Now (with $\|\cdot\|$ denoting the $L^2(\mathbb{R}^+)$ norm)  
\begin{eqnarray}
\left\| \psi(x,t)\,-\,\psi^+_0(x,t)\right\|\,\,\leq\,\, 
\left\| \psi(x,t)-\psi^\varepsilon(x,t)\right\|\,\,\,\,+ \ \ \ \ \ \ \ \ \ \  
\nonumber \\
\, \ \ \ \ \ \ \ \ \ \ \ \ \ \ \ \ \ \ \ \ \ 
\left\| \psi^\varepsilon(x,t)-\psi^\varepsilon_0(x,t)\right\|\,\,+\,\,
\left\| \psi^\varepsilon_0(x,t)-\psi^+_0(x,t)\right\|. 
\label{psilimeps}
\end{eqnarray}
For the first term on the right hand side of \eqref{psilimeps}, notice that 
$\tilde\psi^\varepsilon(x,t):=\psi(x,t)-\psi^\varepsilon(x,t)$ is a $C^\infty(\overline{\Omega})$-solution of 
\eqref{PopEq} satisfying \eqref{PopBC} and \eqref{UnifDec}, as so are both $\psi(x,t)$ and $\psi^\varepsilon(x,t)$. 
Hence $\tilde\psi^\varepsilon(x,t)$ has a $t$-independent $L^2(\mathbb{R}^+)$-norm, 
$\|\tilde\psi^\varepsilon(x,t)\|\equiv c_\varepsilon$, $\forall t$. On the other hand, taking 
$t=t_n$ and passing to the limit as $n\to+\infty$, via \eqref{psi0limn} and \eqref{psi0limeps}, 
\begin{eqnarray}
\left\| \psi(x,t)-\psi^\varepsilon(x,t)\right\|\,\equiv\,
c_\varepsilon\,=\,\lim_{n\to +\infty}\left\| \psi(x,t_n)-\psi^\varepsilon(x,t_n)\right\|\,\,= 
\nonumber \\ 
\, \ \ \ \ \ \ \ \ \ \ \ \ \ \ \ \ \ \ \ \ \ 
\lim_{n\to +\infty}\left\| \psi_0^+(x,t_n)-\psi_0^\varepsilon(x,t_n)\right\|\,=\,
\left\| G_0-G_0^\varepsilon \right\|_{L^2(\mathbb{R})}\,\,<\,\varepsilon. 
\label{l4t1}
\end{eqnarray}
The last limit in \eqref{l4t1} is evaluated similarly to that in \eqref{g0l2lim}, with $\psi_0^+(x,t)$ replaced by 
$\psi_0^+(x,t)-\psi_0^\varepsilon(x,t)$ and $G_0$ by $ G_0-G_0^\varepsilon$. 
This straightforward modification of \eqref{g0l2lim} also immediately implies that for the third term on the right hand 
side of \eqref{psilimeps}, 
\begin{equation}
\left\| \psi^\varepsilon_0(x,t)-\psi^+_0(x,t)\right\|\,\leq\, 
\left\| G_0-G_0^\varepsilon \right\|_{L^2(\mathbb{R})}\,=\,c_\varepsilon\,<\,\varepsilon. 
\label{l4t3}
\end{equation} 
Hence, employing in \eqref{psilimeps} estimates \eqref{l4t1} and  \eqref{l4t3}, 
\begin{equation}
\left\| \psi(x,t)\,-\,\psi^+_0(x,t)\right\|\,\,\leq\,\, 2\varepsilon\,+\,
\left\| \psi^\varepsilon(x,t)-\psi^\varepsilon_0(x,t)\right\|, \ \ \ \forall t>0, \,\forall \varepsilon>0. 
\label{l4t2}
\end{equation} 
Fixing first in \eqref{l4t2} $\varepsilon>0$, and taking the limit-supremum as $t\to+\infty$ using \eqref{psi0limeps} 
yields
\[
\limsup_{t\to+\infty}\left\| \psi(x,t)\,-\,\psi^+_0(x,t)\right\|\,\leq\,\,2\varepsilon, \ \ \forall  \, \varepsilon>0.
\]
Finally, \eqref{psi0lim} follows as $\varepsilon>0$ is arbitrary.  

To complete the proof of the lemma (and with it of Theorem \ref{thm1}) it remains to notice that via the 
transformation \eqref{SLtransform}, \eqref{QCtao}, \eqref{psi0lim} immediately implies that \eqref{gconv2} does hold for 
the whole of $\tau\to+ 0$. 
Hence, by the uniqueness of the limits, for any other subsequence $\tilde\tau_n\to+ 0$ the application of Lemma \ref{lem2} must 
produce the same limit $G_0$ and hence \eqref{gconv1} is also held for the whole of $\tau\to+ 0$, resulting in 
\eqref{gconv}. 
\end{proof}

\section{Concluding remarks} 
We have proved that the solution $\psi(x,t)$ to the inner scattering problem \eqref{PopEq}--\eqref{PopAC} has a searchlight 
asymptotics as $t\to+\infty$ in the form \eqref{psi0} with convergences \eqref{psi0lim} and \eqref{gconv} held. 
The searchlight amplitude $G_0(\eta)$ has certain smoothness and decay properties, as also proved in Theorem \ref{thm1}. 
The resulting analogue of a scattering operator is a well-defined unitary operator, Corollary \ref{cor1}, and provides a full information of how the intensities of the incoming whispering gallery modes transform into the outgoing searchlight amplitudes. 
Still, the searchlight asymptotics as $t\to +\infty$ is non-uniform in the sense that an additional effort would be needed for 
understanding how exactly the solution behaves for large positive $t$ near the boundary $x=0$. 
This would require a separate investigation.


\end{document}